\documentclass[letterpaper,12pt]{article}
\usepackage{tikz}
\usepackage{graphicx}
\usepackage{color}
\usepackage{setspace}
\usepackage{url}
\usepackage{latexsym,amsmath,amssymb,theorem,graphics,epsfig,subfigure}
\usepackage{color}
\usepackage{multirow}
\theoremstyle{break}
\newtheorem{theorem}{Theorem}
\newtheorem{corollary}{Corollary}
\newtheorem{proposition}{Proposition}
\newtheorem{lemma}{Lemma}
\theorembodyfont{\rm}
\newtheorem{definition}{Definition}
\newtheorem{example}{Example}

\newtheorem{remark}{Remark}

\newcommand{\lran}{\longrightarrow}

\def\lint#1{\left[{#1}\right[}
\def\rint#1{\left]{#1}\right]}
\def\cint#1{\left[{#1}\right]}
\def\opint#1{\left]{#1}\right[}
\def\uint{[0,1]}

\def\ouint{\left]0,1\right[}

\def\prooftxt{\mbox{\large\sc proof: }}
\def\konproof{\rm\hspace*{\fill}$\Box$}
\newenvironment{proof}{\par\smallskip\noindent\prooftxt }%
{\konproof\par\vspace*{6pt}}
\newcounter{inst}
\newenvironment{mylist}{\begin{list}{(\roman{enumi})}{\usecounter{enumi}%
                        \setlength{\listparindent}{0pt}%
                        \setlength{\labelsep}{0.4em}%
                        \setlength{\labelwidth}{2.1em}
                        \setlength{\leftmargin}{2em}%
                        \setlength{\itemsep}{-0.8ex}%
                        \setlength{\topsep}{0.4ex}%
                        \setlength{\rightmargin}{0pt}}}{\end{list}}
\date{}
\author{
{\bf Andrea Mesiarov\' a-Zem\' ankov\' a} \\ \ Mathematical Institute, Slovak Academy of Sciences \\
Bratislava, SLOVAKIA \\
zemankova@mat.savba.sk
}

\title{\bf   Ordinal sums of representable uninorms}

\begin{document}

\maketitle \thispagestyle{empty}

\begin{abstract}
We investigate properties of an ordinal sum of uninorms introduced in \cite{genuni} in the case that the summands are proper representable uninorms.
We show sufficient and necessary conditions for a uninorm to be an   ordinal sum of representable uninorms.

\end{abstract}

\section{Introduction}
\label{sec1}

Triangular norms, t-conorms and uninorms (\cite{AFS06,KMP00}) are applied in many domains and therefore  several construction methods for such aggregation functions were developed. Among others, let us recall the construction using the additive generators and the ordinal sum construction.
 A triangular norm  is a binary function $T \colon \uint^2 \lran \uint $ which is commutative,
associative, non-decreasing in both variables and $1$ is its neutral element. Due to the associativity $n$-ary form of any t-norm is uniquely given and  thus it can be extended to an aggregation function working on  $\bigcup_{n\in \mathbb{N}}\uint^n.$
Dual functions to t-norms are t-conorms. A triangular conorm  is a binary function $C \colon \uint^2 \lran \uint $ which is commutative,
associative, non-decreasing in both variables and $0$ is its neutral element.
The duality between t-norms and t-conorms is expressed by the fact that from any t-norm $T$ we can obtain its dual t-conorm $C$ by the equation
$$C(x,y)=1-T(1-x,1-y)$$
and vice-versa.

\begin{proposition}
Let $t\colon \uint\lran \cint{0,\infty}$ be a continuous strictly decreasing function such that $t(1)=0.$ Then the binary  operation  $T \colon \uint^2 \lran \uint $  given by
$$T(x,y)=t^{-1}(\min(t(0),t(x)+t(y)))$$
is a continuous t-norm. The function $t$ is called an \emph{additive generator} of $T.$
\end{proposition}

An additive generator of a continuous t-norm $T$ is uniquely determined up to a positive multiplicative constant. Similarly, an additive generator of a continuous t-conorm $C$ is a continuous strictly increasing function $c\colon \uint\lran \cint{0,\infty}$ such that $c(0)=0.$

Now let us recall an ordinal sum construction for t-norms and t-conorms \cite{KMP00}.

\begin{proposition} Let $K$ be a finite or countably infinite index set and
let $(\opint{a_k,b_k})_{k\in K}$ ($(\opint{c_k,d_k})_{k\in K}$) be a disjoint system of open subintervals of $\uint.$ Let $(T_k)_{k\in K}$ ($(C_k)_{k\in K}$) be
a system of  t-norms (t-conorms). Then
the ordinal sum $T =(\langle a_k, b_k,T_k \rangle \mid  k \in K)$ ($C =(\langle c_k, d_k,C_k \rangle \mid  k \in K)$) given by
$$T(x,y)=\begin{cases} a_k +(b_k-a_k) T_k(\frac{x-a_k}{b_k-a_k},\frac{y-a_k}{b_k-a_k}) &\text{if $(x,y)\in \lint{a_k,b_k}^2 ,$} \\
\min(x,y) &\text{else} \end{cases}$$
and
$$C(x,y)=\begin{cases} c_k +(d_k-c_k) C_k(\frac{x-c_k}{d_k-c_k},\frac{y-c_k}{d_k-c_k}) &\text{if $(x,y)\in \rint{c_k,d_k}^2 ,$} \\
\max(x,y) &\text{else} \end{cases}$$
is a  t-norm (t-conorm). The t-norm $T$ (t-conorm $C$) is continuous if and only if all summands $T_k$ ($C_k$) for $k\in K$ are continuous.
\end{proposition}

More details on t-norms and t-conorms can be found in \cite{AFS06,KMP00}. In order to model bipolar behaviour, uninorms were introduced in \cite{YR96} as binary functions on $\uint$ which are commutative,
associative, non-decreasing in both variables and have a neutral element $e\in \ouint$ (see also \cite{FYR97}). A uninorm can be also taken as a bipolar t-conorm on $\cint{-1,1}$ (see \cite{jacon}), i.e., a bipolar operation that is disjunctive with respect to the neutral point $0$ (i.e., aggregated values diverge from the neutral point). If we take uninorm in a broader sense, i.e., if for a neutral element we have  $e\in \uint,$ then the class of uninorms covers also the class of t-norms and the class of t-conorms. In order the stress that we assume a uninorm with $e\in \opint{0,1}$ we will call such a uninorm \emph{proper}. For each uninorm the value $U(1,0)\in \{0,1\}$ is the annihilator of $U.$ A uninorm is called \emph{conjunctive} (\emph{disjunctive}) if $U(1,0)=0$ ($U(1,0)=1$). Due to the associativity we can uniquely define $n$-ary form of any uninorm for any $n\in\mathbb{N}$ and therefore in some proofs we will use ternary form instead of binary, where suitable.

For each uninorm $U$ with the  neutral element $e\in \uint,$ the restriction of $U$ to $\cint{0,e}^2$ is a t-norm  on $\cint{0,e}^2$ (i.e, a linear transformation of some t-norm $T_U$)
and the restriction of $U$ to $\cint{e,1}^2$ is a t-conorm  on $\cint{e,1}^2$  (i.e, a linear transformation of some t-conorm $C_U$).  Moreover, $\min(x,y)\leq U(x,y)\leq \max(x,y)$ for all
$(x,y)\in \cint{0,e}\times \cint{e,1}\cup \cint{e,1}\times \cint{0,e}.$

\begin{definition}
A uninorm  $U \colon \uint^2 \lran \uint $ is called
\emph{internal} if  $U(x,y)\in \{x,y\}$ for all $(x,y)\in \uint^2.$
\end{definition}

The following easy lemma was shown in \cite{genuni}.

\begin{lemma}
Let $U\colon \uint^2 \lran \uint $  be a uninorm such that $T_U=\min$ and $C_U=\max.$ Then $U$ is internal. \label{lemmm}
\end{lemma}

In the following result we see that from any pair of a t-norm and a t-conorm we can construct a minimal and maximal uninorm with the given
underlying functions.

\begin{proposition}
Let  $T \colon \uint^2 \lran \uint $ be a t-norm and  $C \colon \uint^2 \lran \uint $ a t-conorm and assume $e\in \uint.$ Then the two functions  $U_{\min},U_{\max} \colon \uint^2 \lran \uint $ given by
$$U_{\min}(x,y)=\begin{cases} e\cdot T(\frac{x}{e},\frac{y}{e}) &\text{if $(x,y)\in \cint{0,e}^2,$ } \\
e+ (1-e)\cdot C(\frac{x-e}{1-e},\frac{y-e}{1-e}) &\text{if $(x,y)\in \cint{e,1}^2,$ } \\
\min(x,y) &\text{otherwise}\end{cases}$$ and
$$U_{\max}(x,y)=\begin{cases} e\cdot T(\frac{x}{e},\frac{y}{e}) &\text{if $(x,y)\in \cint{0,e}^2,$ } \\
e+ (1-e)\cdot C(\frac{x-e}{1-e},\frac{y-e}{1-e}) &\text{if $(x,y)\in \cint{e,1}^2,$ } \\
\max(x,y) &\text{otherwise}\end{cases}$$ are uninorms. We will denote the set of all uninorms of the first type by $\mathcal{U}_{\min}$ and of the second type by $\mathcal{U}_{\max}.$
\end{proposition}

Similarly as in the case of t-norms and t-conorms we can construct uninorms using additive generators (see  \cite{FYR97}).

\begin{proposition}
Let $f\colon \uint\lran\cint{-\infty,\infty},$  $f(0)=-\infty,$ $f(1)=\infty$ be a continuous strictly increasing function.
Then a binary function $U \colon \uint^2 \lran \uint $ given by
$$U(x,y)=f^{-1}(f(x)+f(y)),$$ where $f^{-1}\colon \cint{-\infty,\infty}\lran \uint$ is an inverse function to $f,$ is a uninorm, which will be called a \emph{representable} uninorm.
\end{proposition}

Note that if we relax the monotonicity of the additive generator then the neutral element will be lost and by relaxing the condition  $f(0)=-\infty,$ $f(1)=\infty$ the associativity will be lost. In  \cite{uniru} (see also \cite{jacon}) we can find the following result.

\begin{proposition}
Let $U\colon \uint^2 \lran \uint$ be a uninorm continuous everywhere on the unit square expect of the two points $(0,1)$ and $(1,0).$ Then $U$ is representable, i.e., there exists
such a function $u\colon \uint \lran \cint{-\infty,\infty}$ with $u(e)=0,$ $u(0)=-\infty,$ $u(1)=\infty$ that $U(x,y) = u^{-1}(u(x)+u(y)).$
\label{prouni}
\end{proposition}

Thus a uninorm $U$ is representable if and only if it is continuous on $\uint^2\setminus \{(0,1),(1,0)\},$ which
 completely characterizes the set of representable uninorms.

\begin{definition}
We will denote the set of all uninorms $U$ such that $T_U$ and $C_U$ are continuous by $\mathcal{U},$ and the set of all uninorms $V$ such that
$V(x,0)=0$ for all $x\in \lint{0,1}$  and $V(x,1)=1$ for all $x\in \rint{0,1}$  by $\mathcal{N}.$ Further, we will denote by $\mathcal{N}_{\max}$ ($\mathcal{N}_{\min}$) the set of all uninorms $U\in \mathcal{N}$ such that there exists a uninorm $U_1\in \mathcal{U}_{\max}$ ($U_1\in \mathcal{U}_{\min}$) such that
$U=U_1$ on $\opint{0,1}^2.$
\end{definition}

Note that the class of representable uninorms belongs to the intersection $\mathcal{U} \cap \mathcal{N}.$

In the case of t-norms (t-conorms), each continuous t-norm (t-conorm) is an ordinal sum of continuous generated t-norms (t-conorms).
The aim of this paper is the characterization of the uninorms that are ordinal sums of proper representable uninorms. In the following section we will investigate properties of ordinal sums of proper representable uninorms and in Section \ref{sec3} we will completely characterize uninorms which are ordinal sums of proper representable uninorms. We give our conclusions in Section \ref{sec4}.

\section{Ordinal sum of representable uninorms}
\label{sec2}

An ordinal sum of uninorms was introduced in \cite{genuni}.
For any $0\leq c\leq a<b\leq d\leq 1,$  $v\in \cint{a,b}$ and a uninorm $U$  with the neutral element $e\in \uint$ we will use a  transformation $f\colon \uint \lran \lint{c,a} \cup \{v\} \cup \rint{b,d}$
 given by
$$f(x)=\begin{cases} g_1(x) &\text{if $x\in \lint{0,e} ,$ } \\
v  &\text{if $x=e ,$ } \\
g_2(x) &\text{otherwise,} \end{cases}$$ where $g_1\colon \cint{0,e} \lran \cint{c,a}$ and $g_2\colon \cint{e,1} \lran \cint{b,d}$   are  linear increasing functions. Then $f$ is a piece-wise linear isomorphism of $\uint$ to $(\lint{c,a} \cup \{v\} \cup \rint{b,d})$  and a binary function
$U^{c,a,b,d}_v\colon (\lint{c,a} \cup \{v\} \cup \rint{b,d})^2 \lran (\lint{c,a} \cup \{v\} \cup \rint{b,d})$ given by
\begin{equation}U^{c,a,b,d}_v(x,y)=f(U(f^{-1}(x),f^{-1}(y)))\label{unitr}\end{equation} is a uninorm on $(\lint{c,a} \cup \{v\} \cup \rint{b,d})^2.$ Note that in the case
when $a=c$ ($b=d$) we will transform only the part of the uninorm $U$ which is defined on $\cint{0,e}^2$ ($\cint{e,1}^2$).
The function $f$ is piece-wise linear, however, more generally we can use any increasing isomorphic transformation.

\begin{proposition}
Assume $e\in \uint.$
Let $K$ be an index set which is finite or countably infinite and let $(\opint{a_k,b_k})_{k\in K}$   be an  disjoint system of open subintervals (which can be also empty) of $\cint{0,e},$ such that $\bigcup_{k\in K}\cint{a_k,b_k}=\cint{0,e}.$ Similarly, let
 $(\opint{c_k,d_k})_{k\in K}$ be a disjoint system of open subintervals (which can be also empty) of $\cint{e,1},$ such that $\bigcup_{k\in K}\cint{c_k,d_k}=\cint{e,1}.$ Let further these two systems be anti-comonotone, i.e., $b_k\leq a_i$ if and only if  $c_k \geq d_i$
  for all $i,k\in K,$
 and let for all $k\in K$ at least one of
 $\opint{a_k,b_k}$ and $\opint{c_k,d_k}$ be non-empty.  Thus we have for all $k\in K$ that
 $\cint{b_k,c_k}^2\subsetneq \cint{a_k,d_k}^2.$
  Assume
a family of   uninorms  $(U_k)_{k\in K}$ on $\uint^2$ such that if both  $\opint{a_k,b_k}$ and $\opint{c_k,d_k}$ are non-empty then $U_k$ is a proper uninorm, if $\opint{a_k,b_k}$ is non-empty $U_k$ is either a t-norm or a proper uninorm and  if $\opint{c_k,d_k}$ is non-empty then $U_k$ is either a t-conorm or a proper uninorm. Further, let $B$ be the set of all accumulation points of the set $\{b_i\}_{i\in K}$ and $C$ be the set of all accumulation points of the set $\{c_i\}_{i\in K}.$ Assume a function $n\colon B\lran B\cup C$ such that
$n(b_k)\in \{b_k,c_k\}$ for all $b_k\in B.$
Let the ordinal sum $U^e=(\langle a_k,b_k,c_k,d_k,U_k \rangle \mid k\in K)$ be given by
$$U^e(x,y)=\begin{cases}
y &\text{if  $x=e,$  }\\
x &\text{if  $y=e,$  }\\
(U_k)^{a_k,b_k,c_k,d_k}_{v_k} &\text{if  $(x,y)\in (\lint{a_k,b_k}\cup \rint{c_k,d_k})^2,$ }\\
x &\text{if $y\in \cint{b_k,c_k}, x\in \cint{a_k,d_k}\setminus\cint{b_k,c_k},$} \\
y &\text{if $x\in \cint{b_k,c_k}, y\in \cint{a_k,d_k}\setminus\cint{b_k,c_k},$} \\
\min(x,y) &\text{if $(x,y)\in\cint{b_k,c_k}^2\setminus (\opint{b_k,c_k}^2\cup \{(b_k,c_k),(c_k,b_k)\}),$}\\
&\text{ where $b_k\in B,$ $x+y<c_k+b_k,$}\\
\max(x,y) &\text{if $(x,y)\in\cint{b_k,c_k}^2\setminus (\opint{b_k,c_k}^2\cup \{(b_k,c_k),(c_k,b_k)\}),$}\\
&\text{ where $b_k\in B,$ $x+y>c_k+b_k,$}\\
n(b_k)  &\text{if $(x,y)=(b_k,c_k)$ or  $(x,y)=(c_k,b_k),$}
\end{cases}$$ where $v_k=c_k$ ($v_k=b_k$) if there exists an $i\in K$ such that $b_k=a_i$ and $U_i$ is disjunctive (conjunctive)
and $v_k=n(b_k)$ if $b_k\in B,$ and
$(U_k)^{a_k,b_k,c_k,d_k}_{v_{k}}$ is given by the formula \eqref{unitr}.
Then $U^e$ is a uninorm. \label{proorduni}
\end{proposition}

The above ordinal sum is an ordinal sum in the sense of Clifford \cite{cli},
i.e., such where the summand semigroups are not overlapping,
if the predecessor of each summand
which is a nilpotent uninorm is a uninorm with the neutral
element $e_k$ such that $U(x,y)\neq e_k$ holds for all $(x,y)\neq (e_k,e_k)$ (see \cite{genuni}).
Note that here nilpotent uninorm means such uninorm $U$ that $U(x,y)=u$ for some
$x,y\neq u,$ where $u$ is the annihilator of $U,$ and if for two summands in the ordinal sum defined on $\lint{a_{k_1},b_{k_1}}\cup \rint{c_{k_1},d_{k_1}}$ and on $\lint{a_{k_2},b_{k_2}}\cup \rint{c_{k_2},d_{k_2}},$ respectively, we have  $b_{k_1}=a_{k_2}$ (i.e., also
$c_{k_1}=d_{k_2}$) then the first is called a predecessor summand of the second.

\begin{remark}
It is evident that ordinal sum of uninorms $U^e=(\langle a_k,b_k,c_k,d_k,U_k \rangle \mid k\in K)$ is on $\cint{0,e}^2$ equal to an ordinal sum of t-norms, i.e.,
$T_U=(\langle a_k,b_k,T_{U_k} \rangle \mid k\in K)$ and on $\cint{e,1}^2$ to an ordinal sum of t-conorms $C_U=(\langle c_k,d_k,C_{U_k} \rangle \mid k\in K).$
In the case that we assume an ordinal sum of uninorms such that $\bigcup_{k\in K}\cint{a_k,b_k}\neq\cint{0,e}$
($\bigcup_{k\in K}\cint{c_k,d_k}\neq\cint{e,1}$) this can be given by the above ordinal sum, where the missing summands are covered by internal uninorms. Later we will see that this  holds also vice-versa, i.e., if $U$ is an ordinal sum of uninorms and  $T_U=(\langle a_k,b_k,T_k \rangle \mid k\in K)$ and $C_U=(\langle c_k,d_k,C_k \rangle \mid k\in K)$ and  $\lint{a,b}\cup \rint{c,d}$ is a missing summand support, i.e., such that is not covered by
    $\bigcup_{k\in K}\cint{a_k,b_k}\cup\bigcup_{k\in K}\cint{c_k,d_k},$ then $U(x,y)=\min(x,y)$ on $\cint{a,b}^2$ and $U(x,y)=\max(x,y)$ on $\cint{c,d}^2.$ Moreover, similarly as in Lemma \ref{lemmm} we get that $U$ is internal on
   $(\cint{a,b}\cup \cint{c,d})^2.$
\end{remark}

\begin{example}
Assume $U_1 \in \mathcal{U}_{\min}$ and $U_2\in \mathcal{U}_{\max}$ then $U_1$ and $U_2$ are ordinal sums of uninorms,
$U_1=(\langle e,e,e,1,C_U\rangle,\langle 0,e,1,1,T_U\rangle)$  and $U_2=(\langle 0,e,e,e,T_U\rangle,\langle 0,0,e,1,C_U\rangle).$
\end{example}

Assume a uninorm $U\colon \uint^2\lran \uint $ such that
$U=(\langle a_k,b_k,c_k,d_k,U_k \rangle \mid k\in K), $ where all conditions from Proposition \ref{proorduni} are satisfied and
both $\opint{a_k,b_k}$ and $\opint{c_k,d_k}$ are non-empty for all $k\in K.$ We will call such an ordinal sum \emph{complete}.
 Let each $U_k$ for
$k\in K$ be a representable uninorm. Then by Proposition \ref{prouni} the uninorm $U_k$ is continuous on $\uint\setminus \{(0,1),(1,0)\}.$
Since the summand corresponding to $U_k$ acts on $\lint{a_k,b_k}\cup \rint{c_k,d_k}$ uninorm $U$ is continuous  on $(\lint{a_k,b_k}\cup \rint{c_k,d_k})^2$ except the set $\{(f_k(x),f_k(y))\mid U_k(x,y)=e_k\}\cup\{(a_k,d_k),(d_k,a_k)\},$ where $e_k$ is the neutral element of $U_k$ and $f_k$ is the  transformation given by \eqref{unitr} respective to the summand corresponding to $U_k.$ Here let us note that for a representable uninorm $U_k$ there exists a strictly decreasing function  $r_k\colon \opint{0,1} \lran \opint{0,1}$ with $r_k(e_k)=e_k$ such that
$U_k(x,y)=e_k$ if and only if $r_k(x)=y.$

 The ordinal sum construction further implies that for $x\in \{a_k,b_k,c_k,d_k\}$
and $y\in \uint$ we have $U(x,y)\in \{x,y\}.$ Moreover, $U(a_k,y)=a_k$ and $U(d_k,y)=d_k$  for $y\in \opint{a_k,d_k}$ and $U(b_k,y)=b_k$ and $U(c_k,y)=c_k$  for $y\in \opint{b_k,c_k}.$ Since also $U_k(z,e)=\max(z,e)$ for $z>e$  and $U_k(z,e)=\min(z,e)$ for $z<e$ we see that $U$ is continuous on $\{b_k\} \times \cint{a_k,b_k},$  $  \cint{a_k,b_k}\times\{b_k\}, $ $\{c_k\}\times \cint{c_k,d_k},$  $  \cint{c_k,d_k}\times\{c_k\}, $
$\{b_k\} \times \cint{c_k,d_k} \setminus\{(b_k,c_k)\},$ $ \cint{c_k,d_k} \times \{b_k\} \setminus\{(c_k,b_k)\}$ and on
$\{c_k\} \times \cint{a_k,b_k} \setminus\{(c_k,b_k)\},$ $ \cint{a_k,b_k} \times \{c_k\} \setminus\{(b_k,c_k)\}.$ If we summarize this  over all summands for $k\in K$ we obtain the following result.

\begin{proposition}
Assume a uninorm $U\colon \uint^2\lran \uint.$  If $U$ is a complete ordinal sum of   representable uninorms, i.e., $U=(\langle a_k,b_k,c_k,d_k,U_k \rangle \mid k\in K), $ for some suitable systems  $(\opint{a_k,b_k})_{k\in K}$  and  $(\opint{c_k,d_k})_{k\in K}$  and a family of (proper) representable uninorms $(U_k)_{k\in K}$ then there exists a continuous strictly decreasing function $r\colon \uint \lran \uint$ with $r(0)=1,$ $r(e)=e$ and
$r(1)=0$ such that  $U$ is continuous on $\uint\setminus \{(x,r(x))\mid x\in \uint\}.$
Note that $U$ need not be non-continuous on the whole   set $\{(x,r(x))\mid x\in \uint\}.$
\end{proposition}

\begin{example}
Let $U_1,U_2\colon \uint^2\lran \uint$ with $U_1=U_2$ be representable uninorms generated by
$$f=\begin{cases} \ln (2x) &\text{if $x\leq 0$} \\
-\ln(2-2x) &\text{otherwise,} \end{cases}$$ with $U_1(1,0)=1.$ Then the ordinal sum
$U^e=(\langle 0,\frac{1}{4},\frac{3}{4},1,U_1 \rangle,\langle \frac{1}{4},\frac{1}{2},\frac{1}{2},\frac{3}{4},U_2 \rangle),$ with $e=\frac{1}{2}$
is given in the following table

\

\noindent\begin{tabular}{|l||l|l|l|l|}
  \hline
  $x\backslash y$ & $\lint{0,\frac{1}{4}}$ & $\cint{\frac{1}{4},\frac{1}{2}}$ & $\rint{\frac{1}{2},\frac{3}{4}}$ & $\rint{\frac{3}{4},1}$ \\\hline \hline
  \multirow{2}{*}{$\lint{0,\frac{1}{4}}$} & $\frac{x}{4-4y}$  {\scriptsize if $x+y<1$}  &  \multirow{2}{*}{$\max(x,y)$} &   \multirow{2}{*}{$\max(x,y)$} &   \multirow{2}{*}{1-4(1-x)(1-y)} \\
   & $\frac{4x-1+y}{4x}$  {\scriptsize if $x+y>1$ } &  &  &  \\\hline
   \multirow{2}{*}{ $\cint{\frac{1}{4},\frac{1}{2}}$}  &  \multirow{2}{*}{$\min(x,y)$} & $\frac{x-y+\frac{1}{2}}{3-4y}$ {\scriptsize if $x+y<1$} &  \multirow{2}{*}{$\frac{3}{4}-4(\frac{3}{4}-x)(\frac{3}{4}-y)$} & \multirow{2}{*}{$\max(x,y)$} \\
   &  & $\frac{3x+y-\frac{3}{2}}{4x-1}$ {\scriptsize if $x+y>1$} &  &  \\\hline
  \multirow{2}{*}{ $\rint{\frac{1}{2},\frac{3}{4}}$} &  \multirow{2}{*}{$\min(x,y)$} & $4(x-\frac{1}{4})(y-\frac{1}{4})+\frac{1}{4}$ & $ \frac{y-x+\frac{1}{2}}{3-4x}$ {\scriptsize if $x+y<1$} &  \multirow{2}{*}{$\max(x,y)$} \\
   &  &  & $\frac{3y+x-\frac{3}{2}}{4y-1}$ {\scriptsize if $x+y>1$} &  \\\hline
  \multirow{2}{*}{ $\rint{\frac{3}{4},1}$} &  \multirow{2}{*}{$4xy$} &  \multirow{2}{*}{$\min(x,y)$} &  \multirow{2}{*}{$\min(x,y)$} & $\frac{y}{4-4x}$ {\scriptsize if $x+y<1$} \\
&  &  &  & $\frac{4y-1+x}{4y}$ {\scriptsize if $x+y>1$} \\
  \hline
\end{tabular}

\noindent where if $x+y=1$ then $U(x,y)=\frac{1}{2}$ for $(x,y)\in \opint{\frac{1}{4},\frac{3}{4}}$ and otherwise $U(x,y)=\frac{3}{4}.$
Thus here $U^e$ is non-continuous only in the points from the set $\{(x,1-x)\mid x\in \cint{0,\frac{1}{4}}\cup \cint{\frac{3}{4},1} \}.$ Moreover, evidently $U^e \in \mathcal{U} \cap \mathcal{N}.$

\end{example}

Further we will show some general properties of uninorms that we will use later.

\begin{proposition}
Assume a uninorm $U\colon \uint^2\lran \uint$ such that $U\in \mathcal{U}$ and $U\notin \mathcal{N}.$ Then $U$ is an ordinal sum of a uninorm and a non-proper uninorm
(i.e., a t-norm or a t-conorm). \label{pron}
\end{proposition}

\begin{proof}
Assume $U(1,0)=1,$ the case when $U(1,0)=0$ can be shown analogically. Then $U(x,1)=1$ for all $x\in \uint$ and
$U(x,0)=0$ for all $x\in \cint{0,e}.$ If  $U\notin \mathcal{N}$ then there exists a $y\in \opint{e,1}$ such that $U(0,y)=z>0,$ i.e., $z\in \rint{0,y}.$
If $z\leq e$ we get $0=U(0,e)\geq U(0,z)=U(0,0,y)=z$ what is a contradiction, i.e., $z\in \rint{e,y}.$ Moreover,
$U(0,z)=z.$ Since $C_U$ is continuous then for all $x\in \cint{z,1}$ there exists a $u\in \cint{e,1}$ such that $U(z,u)=x.$
Then $$U(0,x)=U(0,z,u)=U(z,u)=x.$$ Thus if $U(0,y)>0$ for some $y\in \opint{e,1}$ then $U(0,y)=y.$ Let $b=\inf\{y\in \cint{e,1}\mid U(0,y)>0\}.$
Then $b$ is an idempotent point of $U,$ since otherwise the continuity of $C_U$ implies existence of $x_1\in \opint{e,b}$ such that
$U(x_1,x_1)=v>b$ which means $$b<v=U(0,v)=U(0,x_1,x_1)=U(0,x_1)=0$$ what is a contradiction.
 Further, if $U(0,y)=y$ for some $y\in \cint{e,1}$ then $U(x,y)=U(x,0,y)=U(0,y)=y$ for all $x\in \cint{0,e}$ and thus
 $U(x,y)=\max(x,y)$ if $(x,y)\in \cint{0,e}\times \rint{b,1} \cup \rint{b,1}\times \cint{0,e}.$

 Since $b$ is an idempotent point. $C_U$ (its transformation onto $\cint{e,1}^2$) is an ordinal sum $C_U=(\langle e,b,C_1 \rangle,\langle b,1,C_2 \rangle)$ and thus $U$ on $\cint{b,1}^2$ corresponds to $C_2$ and $U(x,y)=\max(x,y)$
for all $(x,y)\in \cint{0,b}\times \rint{b,1} \cup \rint{b,1}\times \cint{0,b}.$ Also, $\cint{0,b}^2$ is closed under $U.$ Thus $$U=(\langle 0,e,e,b,U^*\rangle,\langle 0,0,b,1,C_2\rangle),$$ where $U^*$ is uninorm which is a linear transformation of $U$ on $\cint{0,b}^2.$

\end{proof}

\begin{definition}
Let $p$ be a relation on $X\times Y$ and denote $p(x)=\{y\in Y \mid (x,y)\in p\}.$ Then $p$ will be called a \emph{continuous non-increasing pseudo-function} if
\begin{mylist}
\item for all $x_1,x_2\in X,$ $x_1<x_2$ there is $p(x_1)\geq p(x_2),$ i.e, for all $y_1\in p(x_1)$ and all $y_2\in p(x_2)$ we have
$y_1\geq y_2$ and thus $\mathrm{Card}(p(x_1)\cap p(x_2))\leq 1,$
\item for all $x\in X$ and all $y\in Y$ there exist $y_1\in Y$ and $x_1\in X$ such that $(x,y_1)\in p$ and $(x_1,y)\in p,$
\item if $y_1,y_2\in p(x) $ for some $x\in X$ then $y\in p(x)$ for all $y\in \cint{y_1,y_2}.$
\end{mylist}
A relation $p$ is called symmetric if $(x,y)\in p$ if and only if $(y,x)\in p.$
\end{definition}

\begin{remark}
\begin{mylist}
\item
In the case that $U$ is an ordinal sum of   representable uninorms which is not complete, then the function that determine the non-continuity points need not be strictly decreasing, just non-increasing. If there is a non-proper summand such that $a_k=b_k=0$ ($c_k=d_k=1$) then for $r_k$ we have
$r_k(0)=c_k$ ($r_k(1)=b_k$). Further if there  is a non-proper summand, i.e., $a_k=b_k>0$ ($c_k=d_k<1$) for some  $k\in K$ then
$U$ is non-continuous on $\{a_k\}\times \cint{c_k,d_k}$ ($\{c_k\}\times \cint{a_k,b_k}$). Thus in such a case $r_k$  is no longer a function since it contains also some vertical segments, however, $r_K$ is a symmetric continuous non-increasing pseudo-function.
\item If we assume and ordinal sum, where some summands are representable uninorms and some summands are internal uninorms then again we can obtain a symmetric continuous non-increasing pseudo-function $r$  such that $U$ is continuous on  $\uint\setminus \{(x,r(x))\mid x\in \uint\}.$ This follows from the fact that for an internal uninorm $V$ the monotonicity implies existence of a  symmetric continuous non-increasing pseudo-function
    $p_V$ such that $V(x,y)=\max(x,y)$ if $y>p_V(x)$ and $V(x,y)=\min(x,y)$ if $y<p_V(x).$
\end{mylist}
\end{remark}

In the following section we will try to characterize uninorms that are ordinal sums of proper representable uninorms.

\section{Characterization of uninorms that are equal to an ordinal sum of proper representable uninorms}
\label{sec3}

 For a given uninorm $U\colon \uint^2\lran \uint$ and  each $x\in \uint$ we define a function $u_x(z)=U(x,z)$ for $z\in \uint.$ We will start with the following useful result.

\begin{lemma}
Let $U\colon \uint^2\lran \uint$ be a uninorm, $U\in \mathcal{N}\cap \mathcal{U}.$ If $a\in \opint{0,1},$ $a\neq e$ is an idempotent point of $U$
then $u_a$ is non-continuous. \label{lenon}
\end{lemma}

The above result follows from the fact that if $a\neq e,0,1$ is an idempotent element of $U$ then
$e\notin \mathrm{Ran}(u_a).$ Indeed, if $U(a,b)=e$ for some $b\in \uint$ then $e=U(a,a,b)=a$ what is a contradiction.

\begin{proposition}
Let $U\colon \uint^2\lran \uint$ be a uninorm, $U\in \mathcal{N}\cap \mathcal{U}.$ Then if $U(a,b)=e$ for some $a,b\in \uint,$ $a<e$ then
$U$ is continuous on $\uint^2\setminus (\lint{0,a}\cup\rint{b,1})^2.$ \label{procon}
\end{proposition}

\begin{proof} If  $U(a,b)=e$ then $a\notin \{0,1\}.$ Also if $U(a,b)=U(a,c)=e$ then
$b=U(b,a,c)=c,$ i.e., $b=c.$
 First we show that $u_a$ is continuous on $\uint.$ Since
$U$ is monotone and $u_a(0)=0,$ $u_a(1)=1$ the continuity of $u_a$ is equivalent with the equality $\mathrm{Ran}(u_a)=\uint.$
Assume that $\mathrm{Ran}(u_a)\neq\uint,$ i.e., there exists a $c\in \uint$ such that $c\notin \mathrm{Ran}(u_a).$ However, we have
$U(a,b,c)=c,$ i.e., for $z=U(b,c)$ we have $u_a(z)=c$ what is a contradiction.

 Thus $u_a$ and similarly $u_b$ are continuous functions.
Next we will show that for all $x\in \opint{a,b}$ there exists a $v_x\in \uint$ such that $U(x,v_x)=e.$
Assume $f\in \rint{a,e}$ (for $f\in \lint{e,b}$ the proof is analogous).
 Since $T_U$ is continuous and $U(a,f)\leq a,$ $U(f,e)=f$ there exists a $a_f\in \cint{0,e}$ such that
 $U(f,a_f)=a.$ Then $$e=U(a,b)=U(f,a_f,b) $$ and if $w_f=U(a_f,b)$ then $U(f,w_f)=e.$ Summarising we get that all $u_x$ for $x\in \cint{a,b}$
 are continuous. From the previous lemma we see that $a$ and $b$ are not idempotent elements, i.e., there exist $g,h$ with $g<a<b<h$ such that
 $U(g,h)=e,$ i.e., all $u_x$ for $x\in \cint{g,h}$
 are continuous. Now monotonicity of $U$ implies continuity on $\uint^2\setminus (\lint{0,a}\cup\rint{b,1})^2$ (see \cite{KD69}).
\end{proof}

From the previous proposition we see that if  $U\in \mathcal{N}\cap \mathcal{U}$ and $U$ is non-continuous in some point $(c,d)\in \uint^2$
then $u_x$ is non-continuous for all $x\in \cint{0,c}\cup \cint{d,1}.$

%

\begin{lemma}
Assume a uninorm $U\colon \uint^2\lran \uint,$ $U\in \mathcal{U}\cap \mathcal{N}.$   If $a\in \uint$ is an idempotent element of $U$ then
$U$ is internal on $\{a\}\times \uint.$ \label{lemint}
\end{lemma}

\begin{proof} If $a\in \{0,1,e\}$ the result is evident.  Otherwise we will assume $a<e$ (the proof for $a>e$ is analogous). Since $a$ is an idempotent point we have
$U(a,x)=\min(x,y)$ for all $x\in \cint{0,e}.$ From Lemma \ref{lenon} it follows that $u_a$ is non-continuous and
$e\notin \mathrm{Ran}(u_a). $
Assume $y>e$ and let $U(a,y)=v\leq y.$ Then if $v\leq e$ we have $v=U(a,a,y)=U(a,v)\leq a,$ i.e., $v=a.$
Thus if $v>a$ also $v>e.$ Denote $$b=\inf\{y\in \uint\mid U(a,y)>a\}.$$ Then
$U(a,y)>a$ for $y> b$ and $U(a,y)=\min(a,y)$ for $y<b.$ For $v=U(a,y)>a$ we further  have $v=U(a,a,y)=U(a,v).$ Since $U\in \mathcal{U}$
the continuity of $C_U$ ensures for any $y_2>v$ existence of $y_1$ such that $C(v,y_1)=y_2.$ Then $$U(a,y_2)=U(a,v,y_1)=U(v,y_1)=y_2.$$ Summarising,
for all $y>b$ we have $U(a,y)=y$ and for all $x<b$ we have     $U(x,a)=\min(a,x).$ To conclude the proof we have only to check the value
$U(a,b).$ Assume $U(a,b)=c\in \opint{a,b}.$ Then $c\geq e$ and $U(a,c)=c<b,$ what is a contradiction since $U(a,x)=\min(a,x)$ for all $x<b.$
\end{proof}

The above lemma shows, that if we denote the set of all idempotent point of $U$ by $I_U,$ then $U$ restricted to
$I_U^2$ is an internal uninorm.

\begin{lemma}
Each uninorm $U\colon \uint^2\lran \uint,$ $U\in \mathcal{U}$ is continuous in $(e,e).$
\end{lemma}

\begin{proof}
If $T_U$ and $C_U$ are continuous, since $U$ is commutative, we have only to check that for two monotone  sequences $\{a_n\}_{n\in \mathbb{N}}$
and $\{b_n\}_{n\in \mathbb{N}}$ with $\lim\limits_{n\lran \infty} a_n =e =\lim\limits_{n\lran \infty} b_n$ and $a_n< e,$ $b_n> e$ for $n\in \mathbb{N}$ there is $\lim\limits_{n\lran \infty} U(a_n,b_n)=e.$ However, monotonicity gives us $a_n\leq U(a_n,b_n)\leq b_n$ and thus
$e\leq \lim\limits_{n\lran \infty} U(a_n,b_n)\leq e.$
\end{proof}

From now on we will investigate uninorms such that  there exists a continuous strictly decreasing function $r\colon \uint \lran \uint$ with $r(0)=1,$ $r(e)=e$ and
$r(1)=0$ such that  $U$ is continuous on $\uint\setminus \{(x,r(x))\mid x\in \uint\}.$
 Then  if $U$ is non-continuous only in $(0,1),(1,0)$ the uninorm $U$ is representable. Thus if $U$ is not representable then due to the Proposition \ref{procon} there exist $a,b\in \uint $ such that
$U$ is non-continuous in all points of $\{(x,r(x))\mid x\in \cint{0,a}\cup \cint{b,1}\},$ where $a>0$ and $b<1.$

Each continuous t-norm (t-conorm) is equal to an ordinal sum of continuous Archimedean  t-norms (t-conorms). Note that a continuous t-norm (t-conorm) is Archimedean if and only if it has only trivial idempotent points $0$ and $1.$ A continuous Archimedean t-norm $T$ (t-conorm $C$) is either strict, i.e.,
strictly increasing on $\rint{0,1}^2,$ (on $\lint{0,1}^2$) or nilpotent, i.e.,  there exist $(x,y)\in \opint{0,1}^2$ such that $T(x,y)=0$
($C(x,y)=1$).

For  Archimedean underlying functions we can show the following result.

\begin{proposition}
Assume a uninorm $U\colon \uint^2\lran \uint,$ $U\in \mathcal{U}\cap \mathcal{N}.$
If $T_U$ and $C_U$ are Archimedean then either $U$  is a representable uninorm or $U\in \mathcal{N}_{\min}\cup \mathcal{N}_{\max}$.  \label{sumstrp}
\end{proposition}

\begin{proof}
If for all $x\in \opint{0,1}$ there exists a $y\in \cint{0,1}$ such that $U(x,y)=e$ then since $T_U$ and $C_U$ are continuous and Archimedean  $U$ is continuous on $\cint{0,1}^2$ except points $(0,1)$ and $(1,0)$ and thus $U$ is a representable uninorm.
 If for any $x\in \opint{0,1},$ $x\neq e$ there exists a $y\in \cint{0,1}$ such that $U(x,y)=e$ then also $U(x,x,y,y)=e$ and since $T_U$ and $C_U$
 are continuous and Archimedean  Proposition \ref{procon} implies that for all $x\in \opint{0,1}$ there exists a $y\in \cint{0,1}$ such that $U(x,y)=e$ and thus $U$
 is representable.
 Therefore we will suppose that $U(x,y)=e$ for $(x,y)\in \cint{0,1}^2$ if and only if $x=y=e.$

If $U(x,y)=x=U(x,e)$ (or similarly if $U(x,y)=y$) for some $x,y\in \opint{0,1},$ $x<e$ and $y>e$ then
 $U(x,\underbrace{y,\ldots,y}_{n\text{-times}})=x$ for all $n\in \mathbb{N}$ and since $C_U$ is continuous and  Archimedean we have
 $U(x,z)=x$ for all $z\in \lint{e,1}.$ Further, since $T_U$ is  continuous and  Archimedean for all $0<q\leq x$ there exists
 $x_1\in \cint{0,e}$ such that $U(x,x_1)=q$ and thus $U(q,z)=U(x_1,x,z)=U(x_1,x)=q$ for all $z\in \lint{e,1}.$ Let $$b=\inf\{x\in\cint{0,e}\mid
 U(x,\frac{3}{4})>x  \}.$$ Then $U(x,y)=x$ for all $x<b$ and $y\in \lint{e,1}$ and $U(x,y)>x$ for all $x>b$ and $y\in \lint{e,1}.$
 If $b$ is not an idempotent point then since $T_U$ is continuous and  Archimedean there exists $b_1,$ $b<b_1<e$ such that
 $U(b_1,b_1)=v<b.$  Then for a $y\in \lint{e,1}$ and $U(b_1,y)=w>b_1$ we have $ U(b_1,w) = U(b_1,b_1,y)=U(v,y)=v=U(b_1,b_1)$  
 what is possible only if there is an idempotent point in $\cint{b_1,w}.$ However, then $$b>v=U(b_1,b_1)=b_1>b$$ what is a contradiction.  
 Thus $b$ is an idempotent point. Since $T_U$ and $C_U$ are Archimedean we have $b\in \{0,e\}.$ Thus we get that  either $U(x,y)=x$ for all $(x,y)\in \opint{0,e}\times \opint{e,1},$
 i.e., $U\in \mathcal{N}_{\min},$  or
 $U(x,y)=y$ for all $(x,y)\in \opint{0,e}\times \opint{e,1},$ i.e., $U\in \mathcal{N}_{\max},$  or there is $U(x,y)\in \opint{x,y}$ for all $(x,y)\in \opint{0,e}\times \opint{e,1}.$
 From now on we will suppose that $U(x,y)\in \opint{x,y}$ for all $(x,y)\in \opint{0,e}\times \opint{e,1}.$

 Take any $(x,y)\in \opint{0,e}\times \opint{e,1}$ and then $U(x,y)=c\in \opint{0,e}.$ Without loss of generality assume $c<e$
 (the case when $c>e$ is analogous). Then since $x<c<e$ and $T_U$ is continuous and Archimedean  there exists a $x_1 \in \opint{0,e}$ such that
 $U(c,x_1)=x.$ Then $$c=U(x,y)=U(c,x_1,y)$$ which is a contradiction if $U(x_1,y)\in \opint{e,1}.$ Since $U(x_1,y)\leq y$ and 
 $U(x_1,y)\neq e$   we have $x_1\leq U(x_1,y)=z<e.$ Then  since $z\in \opint{0,e}$ and $T_U$ is Archimedean the  equality $c=U(c,z)$ implies that $z$ is an idempotent point
  what is  a contradiction. Summarising, $U$  is either a representable uninorm or $U\in \mathcal{N}_{\min}\cup \mathcal{N}_{\max}$.
\end{proof}

\begin{corollary}
Assume a uninorm $U\colon \uint^2\lran \uint,$ $U\in \mathcal{U}\cap \mathcal{N}$ and let  there exists a continuous strictly decreasing function $r\colon \uint \lran \uint$ with $r(0)=1,$ $r(e)=e$ and
$r(1)=0$ such that  $U$ is continuous on $\uint\setminus \{(x,r(x))\mid x\in \uint\}.$
Then if there exist $a\in \cint{0,e}$ and $b\in \cint{e,1}$ such that $U$ is an Archimedean (i.e., nilpotent or strict) t-norm on $\cint{a,e}^2$ and $U$ is an Archimedean (i.e., nilpotent or strict) t-conorm on $\cint{e,b}^2$ then $U$ on $\cint{a,b}^2$ is a representable uninorm.  \label{sumstr}
\end{corollary}

\begin{proof}
Since $U$ is a t-norm on $\cint{a,e}^2$ and a t-conorm on $\cint{e,b}^2$ then $a$ and $b$ are idempotent points of $U$ and $U$ is closed on $\cint{a,b}^2,$ i.e., $U$ on   $\cint{a,b}^2$ is isomorphic to a uninorm which we denote by $U^*.$ The previous proposition implies that
either $U^*$ is a  representable uninorm or $U^*\in \mathcal{N}_{\min}\cup \mathcal{N}_{\max}.$ However, if $U^*\in \mathcal{N}_{\min}$
 then $U^*$ is non-continuous in all points from the set $\opint{0,e}\times \{1\},$ i.e., $r$ is not strictly decreasing what is a contradiction. Similarly, if  $U^* \in \mathcal{N}_{\max}$ then $U^*$ is non-continuous in all points from the set $\opint{e,1}\times \{0\}.$
 Thus $U^*$ is representable.
\end{proof}

Before we introduce another result we recall  the claim of \cite[Theorem 5.1]{ordut}.
Here $\mathcal{U}(e)=\{U\colon \uint^2 \lran \uint \mid U \text{ is associative, non-decreasing, with the neutral element } e\in \uint\}.$ Thus $U\in \mathcal{U}(e)$ is a uninorm if it is commutative.

\begin{theorem}
Let $U\in \mathcal{U}(e)$ and $a,b,c,d \in \uint,$ $a\leq b\leq e\leq  c\leq d$ be such that $U\vert_{\cint{a,b}^2}$ is associative, non-decreasing, with the neutral element $b$ and
$U\vert_{\cint{c,d}^2}$ is associative, non-decreasing, with the neutral element $c.$ Then the set $(\cint{a,b}\cup \cint{c,d})^2$ is closed under $U.$ \label{th51}
\end{theorem}

Now we can show the following.

\begin{proposition}
Assume a uninorm $U\colon \uint^2\lran \uint,$ $U\in \mathcal{U}\cap \mathcal{N}$ and let  there exists a continuous strictly decreasing function $r\colon \uint \lran \uint$ with $r(0)=1,$ $r(e)=e$ and
$r(1)=0$ such that  $U$ is continuous on $\uint\setminus \{(x,r(x))\mid x\in \uint\}.$
Then \begin{mylist} \item if $a,b \in \cint{0,e}$ are idempotent elements such that $U(x,x)<x$ for all $x\in \opint{a,b}$ then also $c=r(b)$ and $d=r(a)$ are idempotent elements and $U(y,y)>y$  for all $y\in \opint{c,d},$ 
\item if $c,d \in \cint{e,1}$ are idempotent elements such that $U(y,y)>y$  for all $y\in \opint{c,d}$ then also $b=r(c)$ and $a=r(d)$ are idempotent elements and $U(x,x)<x$ for all  $x\in \opint{a,b}.$ 
\end{mylist} \label{propair}
\end{proposition}

\begin{proof} We will show only the first part, the second part is analogous.
Let  $a,b \in \cint{0,e}$ be idempotent elements of $U$ such that $U(x,x)<x$ for all $x\in \opint{a,b}$ and let $c=r(b)$ and $d=r(a).$ Let  $g$ be the smallest idempotent element of $U$ such that
$g\geq d.$ Then according to Theorem \ref{th51} interval $\cint{a,g}^2$ is closed under $U,$ i.e., it is a linear transformation of some uninorm $U^*,$ $U^*\in \mathcal{U}.$
 If     $U^*\in \mathcal{N}$ then $U$ is non-continuous in $(a,g)$ which means that $g=d.$ If  $U^*\notin \mathcal{N}$ then Proposition \ref{pron} implies that $U^*$ is an ordinal sum of uninorm and a non-proper uninorm and since $U$ is non-continuous
 in $(a,d)$ where $d\leq g$ we have $U(a,z)<e$ for $z<d$ and $U(a,z)>e$ for $z>d,$   
  i.e., $U^*$ is an ordinal sum of a uninorm and a t-conorm and $d$ is an idempotent point, i.e., $d=g.$ Thus
 in all cases $d$ is an idempotent element of $U.$ 
 
 Further, $u_b$ is non-continuous exactly in the one point $x=c $ and since $b$ is idempotent Lemma \ref{lemint} implies $U(b,x)=\min(x,b)$ for $x< c$ and $U(b,x)=x$ for $x>c,$    $U(b,c)\in \{b,c\}.$ If $U(b,c)=b$ then also $U(b,c,c)=b$ which implies $U(c,c)\leq c,$ i.e., $c$ is an idempotent point of $U.$ Assume $U(b,c)=c.$ Then for $x\in \opint{e,c}$ we have $$c=U(b,c)=U(b,x,c)=U(x,c)$$ which means that there is an idempotent point
 in $\cint{x,c}.$ Since $C_U$ is continuous, i.e., the set of all idempotent points is closed we see that $c$ is an idempotent point of $U.$
 Thus both $c$ and $d$ are idempotent points. 
 
 Assume that $h\in \opint{c,d}$ is an idempotent point. Then similarly as above we can show that $r(h)$ is also an idempotent point of $U$ and $a=r(d)<r(h)<r(c)=b,$ i.e., there is an idempotent point between $a$ and $b$ what is a contradiction.
\end{proof}

\begin{definition}
An internal uninorm $U\colon \uint^2\lran \uint$ will be called \emph{s-internal} if there exists a continuous and strictly decreasing function $v_{U}\colon \uint \lran \uint$ such that  $U(x,y)=\min(x,y)$ if $y<v_U(x)$ and  $U(x,y)=\max(x,y)$ if $y>v_U(x).$
\end{definition}

\begin{proposition}
Assume a uninorm $U\colon \uint^2\lran \uint,$ $U\in \mathcal{U}\cap \mathcal{N}$ and let  there exists a continuous strictly decreasing function $r\colon \uint \lran \uint$ with $r(0)=1,$ $r(e)=e$ and
$r(1)=0$ such that  $U$ is continuous on $\uint\setminus \{(x,r(x))\mid x\in \uint\}.$ Then $U$ is an ordinal sum of representable uninorms, i.e.,  $U=(\langle a_m,b_m,c_m,d_m,U_m \rangle \mid m\in M),$ where
 $(\opint{a_m,b_m})_{m\in M}$ and $(\opint{c_m,d_m})_{m\in M}$ are two
  anti-comonotone systems of disjoint non-empty open intervals such that
  $\bigcup_{m\in M}\cint{a_m,b_m}=\cint{0,e}$ and   $\bigcup_{m\in M}\cint{c_m,d_m}=\cint{e,1},$  and
  $(U_m)_{m\in m}$ is a family of   (proper) representable uninorms and s-internal uninorms on $\uint^2.$  \label{charpro}
\end{proposition}

\begin{proof}
Since $T_U$ and $C_U$ are continuous the set of idempotent elements $I_U$ of $U$ is closed and thus $\cint{0,e}\setminus I_U = \bigcup\limits_{m\in M} \opint{a_m,b_m}$ and
$\cint{e,1}\setminus I_U = \bigcup\limits_{l\in L} \opint{c_l,d_l}$ for some countable index sets $M,L$ and two systems of open non-empty disjoint intervals $(\opint{a_m,b_m})_{m\in M}$
and $(\opint{c_l,d_l})_{l\in L}.$ From Proposition \ref{propair} it follows that each interval $\opint{a_m,b_m}$ can be paired with the interval $\opint{c_l,d_l} $ for some $l\in L$ such that $r(a_m)=d_l$ and
$r(b_m)=c_l $ and vice-versa, i.e., we can set $L=M$ and obtain two anti-comonotone systems of open non-empty disjoint intervals   $(\opint{a_m,b_m})_{m\in M}$ and $(\opint{c_m,d_m})_{m\in M},$ where
$\opint{a_m,b_m} \subset \cint{0,e}$ and  $\opint{c_m,d_m} \subset \cint{e,1}$ for all $m\in M.$ Since $a_m,b_m,c_m,d_m$ are idempotent points, Lemma \ref{lemint} and monotonicity of $U$ implies that
$U(x,y)=y$ if $x\in \cint{b_m,c_m}$ and $y\in \cint{a_m,d_m} \setminus \cint{b_m,c_m}.$

 Further,
$(\cint{a_m,b_m}\cup \cint{c_m,d_m})^2$ is closed under $U$ and $U$ on $\cint{a_m,b_m}^2$ is a continuous Archimedean t-norm and $U$ on $\cint{c_m,d_m}^2$ is a continuous Archimedean  t-conorm.   In order to use backward transformation inverse to \eqref{unitr}
we have only to show that $$\mathrm{Card}(\mathrm{Ran}(U\vert_{\lint{a_m,b_m}\cup \rint{c_m,d_m}})\cap \cint{b_m,c_m})<2.$$  Assume $U(x_1,y_1)=f$  for some
$x_1\in \lint{a_m,b_m},$ $y_1 \in \rint{c_m,d_m}$ and $f\in \cint{b_m,c_m}.$ Then for any $z\in \cint{b_m,c_m}$ we have $U(f,z)=U(x_1,y_1,z)=U(x_1,y_1)=f.$ Thus $f$ is the annihilator of $U$ on $\cint{b_m,c_m},$
i.e., $f=U(b_m,c_m).$ Now if we transform $U$ on   $(\lint{a_m,b_m}\cup \{U(b_m,c_m)\} \cup\rint{c_m,d_m})^2$  using $f^{-1},$ where $f$ is given in \eqref{unitr}, where
$c=a_m,$ $a=b_m,$ $v=U(b_m,c_m),$ $b=c_m$ and $d=d_m$ and $e\in \ouint$ we obtain a uninorm $U_m$ on $\uint^2$ with the neutral element $e$ such that $T_{U_m}$ and $C_{U_m}$ are Archimedean and
$U_m\in \mathcal{N} \cap \mathcal{U}.$ Then by Proposition \ref{sumstr} the uninorm $U_m$ is representable.

If $\bigcup\limits_{m\in M} \cint{a_m,b_m}=\cint{0,e}$ and $\bigcup\limits_{m\in M} \cint{c_m,d_m}=\cint{e,1}$ the proof is finished. In the opposite case we have
$\cint{0,e}\setminus \bigcup\limits_{m\in M} \cint{a_m,b_m} = \bigcup\limits_{o\in O} \opint{g_o,h_o},$ where   $(\opint{g_o,h_o})_{o\in O}$ is a system of non-empty open intervals, i.e., $O$ is a countable index set.
Then we have $\cint{e,1}\setminus \bigcup\limits_{m\in M} \cint{c_m,d_m} = \bigcup\limits_{o\in O} \opint{r(h_o),r(g_o)}.$ The set $(\lint{g_o,h_o} \cup \{U(h_o,r(h_o))\} \cup \rint{r(h_o),r(g_o)})^2$ is closed under $U$ and thus it is isomorphic to some uninorm $U_o$ such that $T_{U_o}=\min$ and $C_{U_o}=\max.$ Thus by Lemma \ref{lemmm} $U_o$ is internal. Moreover, since $r$ is continuous and strictly decreasing there exists a continuous and strictly decreasing function $v_{U_o}\colon \uint \lran \uint$ such that  $U_o(x,y)=\min(x,y)$ if $y<v_{U_o}(x)$ and  $U_o(x,y)=\max(x,y)$ if $y>v_{U_o}(x),$ i.e., $U_o$ is an s-internal uninorm.
\end{proof}

\begin{corollary}
A uninorm $U\colon \uint^2\lran \uint,$ $U\in \mathcal{U}\cap \mathcal{N}$ is a complete ordinal sum of representable and s-internal uninorms  if and only if there exists a continuous strictly decreasing function $r\colon \uint \lran \uint$ with $r(0)=1,$ $r(e)=e$ and
$r(1)=0$ such that  $U$ is continuous on $\uint\setminus \{(x,r(x))\mid x\in \uint\}.$
\end{corollary}

\begin{corollary}
A uninorm $U\colon \uint^2\lran \uint,$ $U\in \mathcal{U}\cap \mathcal{N}$ is a complete ordinal sum of representable uninorms if and only if there exists a continuous strictly decreasing function $r\colon \uint \lran \uint$ with $r(0)=1,$ $r(e)=e$ and
$r(1)=0$ such that  $U$ is continuous on $\uint\setminus \{(x,r(x))\mid x\in \uint\}$ and $U$ has countably many idempotent points.
\end{corollary}

This result follows from the fact that if there are countably many idempotent points then there is no interval of idempotent points, i.e., $\bigcup\limits_{m\in M} \cint{a_m,b_m}=\cint{0,e}$ and $\bigcup\limits_{m\in M} \cint{c_m,d_m}=\cint{e,1}.$ On the other hand, if $\bigcup\limits_{m\in M} \cint{a_m,b_m}=\cint{0,e}$ and $\bigcup\limits_{m\in M} \cint{c_m,d_m}=\cint{e,1}$ then idempotent points are only $a_m,b_m,c_m,d_m$ for $m\in M$ and since
$M$ is countable also the set of idempotent points is countable.

Finally, let us note that if we have a uninorm $U\colon \uint^2\lran \uint,$ $U\in \mathcal{U}$ and $U\notin \mathcal{N}$ then according to Proposition \ref{pron} the uninorm $U$ is an ordinal sum of a uninorm and a t-norm (t-conorm). This means that $u_1$ ($u_0$) is non-continuous in some point $x>0$ ($x<1$)
which means that there cannot exist a  continuous strictly decreasing function $r\colon \uint \lran \uint$ with $r(0)=1,$ $r(e)=e$ and
$r(1)=0$ such that  $U$ is continuous on $\uint\setminus \{(x,r(x))\mid x\in \uint\}.$

\section{Conclusions}
\label{sec4}

In this paper we have shown that a uninorm is equal to a complete ordinal sum of  representable uninorms and s-internal uninorms if   and only if there exists a continuous strictly decreasing function $r\colon \uint \lran \uint$ with $r(0)=1,$ $r(e)=e$ and
$r(1)=0$ such that  $U$ is continuous on $\uint\setminus \{(x,r(x))\mid x\in \uint\}.$ Moreover, such a uninorm $U$ is a complete  ordinal sum  of representable uninorms if the  set of all idempotent elements of $U$ is countable. We conjecture that a similar result can be shown for all uninorms, where $T_U$ and $C_U$ are continuous. In such a case we conjecture that the set of all points of non-continuity is characterized by a symmetric continuous non-decreasing pseudo-function and each such a uninorm can be decomposed into ordinal sum of representable uninorms, continuous Archimedean t-norms, t-conorms and internal uninorms. However, any uninorm $U\in \mathcal{N}_{\min}$ such that $C_U$ has no non-trivial elements
is irreducible with respect to the ordinal sum construction, i.e., can be expressed only as a trivial ordinal sum with summand on $(\lint{0,e}\cup \rint{e,1})^2,$ and thus modification of the ordinal sum construction, such where summands will be defined on $(\opint{a_m,b_m}\cup \opint{c_m,d_m})^2$
should be assumed in this case. However, we leave this research for future work.

 \vskip0.3cm \noindent{\bf \large Acknowledgement} \hskip0.3cm  This work was supported by grant  VEGA 2/0049/14
 and Program Fellowship of SAS.

\end{document}